\documentclass[11pt,oneside]{article}
\usepackage[T1]{fontenc}
\usepackage{lmodern}
\usepackage[letterpaper,margin=1in]{geometry}
\usepackage[numbers,sort&compress]{natbib}
\usepackage[hidelinks]{hyperref}
\usepackage{amsmath,amssymb,amsthm,mathtools}
\usepackage{microtype,graphicx,booktabs,enumitem,xcolor,placeins}
\hypersetup{bookmarksdepth=2,pdftitle={The Robust Newton Method for Polynomials and its Polynomiography},pdfauthor={Bahman Kalantari}}
\setlist[enumerate]{itemsep=3pt,topsep=5pt}
\theoremstyle{plain}
\newtheorem{theorem}{Theorem}[section]
\newtheorem{lemma}[theorem]{Lemma}
\newtheorem{proposition}[theorem]{Proposition}
\newtheorem{corollary}[theorem]{Corollary}
\theoremstyle{definition}
\newtheorem{definition}[theorem]{Definition}

\theoremstyle{remark}
\newtheorem{remark}[theorem]{Remark}
\newcommand{\C}{\mathbb C}
\newcommand{\R}{\mathbb R}

\newcommand{\dimH}{\operatorname{dim}_{\mathrm H}}
\newcommand{\Crit}{\operatorname{Crit}_*}
\DeclareMathOperator{\Rea}{Re}
\DeclareMathOperator{\Ima}{Im}

\begin{document}
\title{The Robust Newton Method for Polynomials and its Polynomiography}
\author{Bahman Kalantari\\[3pt]
\small Emeritus Professor of Computer Science, Rutgers University\\
\small Piscataway, New Jersey, USA\\
\small \texttt{kalantari@cs.rutgers.edu}}
\date{}
\maketitle
\begin{abstract}
We present the Robust Newton Method (RNM) for a complex polynomial
$p$ of degree $n\geq2$, with continuous map
$F(z)=z-p(z)\overline{p'(z)}/(9A(z)^2)$, where
$A(z)=\max_{0\leq j\leq n}|p^{(j)}(z)|/j!$.
Its fixed points are the roots and critical points; elsewhere it
strictly decreases $|p|$ along the Newton direction.
The Geometric Modulus Principle inspires the construction and
supplies an escape step at nonroot critical points.
Every root has an open basin. When nonroot critical points are
simple, almost every seed converges to a root, and root-basin
boundaries lie in countably many Lipschitz curves with Hausdorff
dimension one when nonempty. These conclusions apply after any
sufficiently small nonzero linear perturbation $p(z)+\delta z$.
Every RNM orbit, with or without escapes, converges to a root or
critical point. A hybrid enforcing one RNM update every $T$
iterations retains convergence and reaches
$|p(z)p'(z)|\leq\varepsilon$ in $O(T/\varepsilon^2)$ updates,
with constants depending on the polynomial and seed; $T=1$
gives RNM. Orbits approaching simple roots eventually reach
Smale's $\alpha$-certification, permitting a permanent Newton
switch and $O(\log\log(1/\varepsilon))$ further steps for absolute
root error $\varepsilon$. For cubics, deterministic critical-point
initialization yields an explicit preparatory bound, retained by
an RNM-guided Newton refinement. Illustrative computations cover
cubics and degree-ten polynomials. For $p(z)=z^n-1$, the basins
are exactly the open Voronoi regions, with boundary points
converging to the origin. Polynomiographs illustrate basin
geometry and iteration-based coloring.
\end{abstract}


\section{Introduction}\label{sec:intro}
Newton's method is one of the most useful algorithms for approximating a zero of a polynomial. For
\begin{equation}\label{eq:poly}
 p(z)=a_nz^n+\cdots+a_1z+a_0,\qquad a_n\ne0,\quad n\geq2,
\end{equation}
its iteration function is $N(z)=z-p(z)/p'(z)$ for $p'(z)\ne0$. Throughout, $N^{\circ s}$ denotes the $s$-fold composition of $N$ with itself, with $N^{\circ0}(z)=z$; the same convention applies to other iteration maps. Near a simple root the method converges quadratically. Its global behavior can be much more complicated: a seed may approach a pole, belong to a nonroot attracting cycle, or lie on a complicated basin boundary. For example, $p(z)=z^3-2z+2$ has the Newton cycle $0\mapsto1\mapsto0$. The dynamics of Newton and other rational maps are developed in \cite{Beardon,Milnor}. Prescribed-seed methods for approximating all polynomial roots by Newton iteration are studied in \cite{HSS}.

The \emph{Robust Newton Method} (RNM), introduced in \cite{RNM2020}, modifies the step length while preserving the Newton direction at noncritical points. Write
\begin{equation}\label{eq:A}
 a_j(z)=\frac{p^{(j)}(z)}{j!},\qquad
 A(z)=\max_{0\leq j\leq n}|a_j(z)|,\qquad
 h(z)=\frac{1}{9A(z)^2}.
\end{equation}
Here $a_j(z)$ denotes a normalized derivative, whereas the coefficients in \eqref{eq:poly} are constants; in particular, $a_n(z)=a_n$. The continuous RNM map is
\begin{equation}\label{eq:F}
 F(z)=z-h(z)p(z)\overline{p'(z)}.
\end{equation}
It is defined everywhere because $A(z)\geq|a_n|>0$. Under the identification $\C\simeq\R^2$,
\begin{equation}\label{eq:gradient}
 \nabla |p(z)|^2=2p(z)\overline{p'(z)},
 \qquad F(z)=z-\frac{h(z)}2\nabla|p(z)|^2.
\end{equation}
Thus RNM is an explicitly scaled gradient step. Its scale is selected from the normalized derivatives rather than by a line search.

Two different operations must be distinguished. The continuous map $F$ fixes every root and every critical point. A second map, denoted by $E$, agrees with $F$ at noncritical points and uses higher derivatives to leave an exact nonroot critical point. The latter map strictly decreases $|p|$ at every nonroot point but is discontinuous at the points where it differs from $F$. Escaping an exact critical point does not exclude an orbit that approaches a critical point without reaching it. In this paper, \emph{global convergence} means convergence of the iterates to a single root or critical point, unless root convergence is explicitly asserted.

The construction is inspired by the Geometric Modulus Principle (GMP) \cite{GMP}; the modulus-reduction estimate supplies an explicit step length. Here ``robust'' refers to this globally defined descent construction. The preparatory iteration can be slow near critical points and is not uniformly fast over seeds. Selecting descent steps and switching to local Newton iteration are standard strategies \cite[Chapter~9]{BoydVandenberghe2004}. The specific contributions here are the explicit polynomial-dependent step, its certified consequences, and the geometry of its basins.

The earlier preprint \cite{RNM2020} introduced RNM and its modulus-reduction and stationarity estimates, discussed root-or-critical convergence, critical-point initialization for cubics, and switching to Smale-certified Newton iteration. It also treated the quadratic Voronoi geometry. The present treatment distinguishes the continuous and escape maps, completes the full-orbit convergence arguments, gives safeguarded and RNM-guided acceleration and an explicit cubic certification bound, and proves the all-degree Voronoi and general basin-geometry theorems.

Sections~\ref{sec:construction} and~\ref{sec:convergence} develop the GMP-based step, quantitative descent, convergence, root convergence below the least nonroot critical modulus, and openness of root basins. Section~\ref{sec:hybrid} allows modulus-decreasing Newton trials between RNM updates. At least one escape-map update in every $T$ iterations guarantees root-or-critical convergence and an $O(T/\varepsilon^2)$ stationarity bound. Section~\ref{sec:certification} supplies a globally defined preparation for Smale's pointwise certification. Every orbit approaching a simple root eventually passes the test; its entry time can depend on the seed but not on the final accuracy. We give a quantitative bound on sublevel sets without critical points and an example showing why limiting-root data alone cannot control entry time. Throughout, \emph{Newton--RNM hybrid} means interleaved updates, while \emph{RNM--Newton} means RNM-based preparation followed by a permanent Newton switch.

Section~\ref{subsec:cubic} treats the classical problem of solving cubic equations, giving a deterministic initialization and a conservative bound for reaching certified Newton convergence. Cubics also arise in computer graphics, including continuous collision detection and geometric curve constructions \cite{Blinn2006,Yuksel2022}. Section~\ref{subsec:rnm-guided} gives RNM-guided Newton for arbitrary degree: its trial steps satisfy RNM's decrease estimate, retaining the stationarity bound and the cubic preparation bound. The computational comparisons are illustrative; the main results are theoretical.

Section~\ref{sec:unity} proves that the $F$-basins for $z^n-1$ are exactly the open Voronoi regions. For a polynomial with at least three distinct roots, no holomorphic rational iteration function can have a root basin equal to that root's Voronoi region \cite{Voronoi2011}. RNM lies outside this class because it uses conjugation and derivative moduli. The restriction to three roots matters: Newton iteration for $z^2-1$ already has the two open half-planes as root basins. Section~\ref{sec:geometry} proves that, when nonroot critical points are simple, their basins lie in countably many Lipschitz curves and form exactly the union of root-basin boundaries. The proof permits ties in the maximum defining $A(z)$, using a Lipschitz stable-manifold theorem and the real rational branches of $F$. A small linear perturbation gives simple roots and critical points.

These results describe \emph{polynomiography}, the visualization of polynomial root-finding processes and their dynamics \cite{Polynomiography2005,Book2008,Invitation2025}. Earlier RNM images and variants appear in \cite{RNM2020,MannRNM,SRNM}; \emph{Invitation to Polynomiography via ChatGPT: For Teachers, Students and Artists} \cite{Invitation2025} discusses their exploration and artistic use. Section~\ref{sec:figures} gives arithmetic costs and polynomiographs illustrating the proved geometry. Section~\ref{sec:conclusion} discusses truncated RNM and further questions.

All theorems concern exact arithmetic. Since RNM is not a holomorphic rational map, it is outside the algorithmic class in McMullen's impossibility theorem for generally convergent purely iterative rational algorithms in degree at least four \cite{McMullen}. That work also constructs generally convergent rational algorithms for cubics.

\section{The RNM step and modulus reduction}\label{sec:construction}
\subsection{Geometric descent directions}
We first recall the Geometric Modulus Principle \cite{GMP}, in the local form underlying the construction.
\begin{theorem}[Geometric Modulus Principle]\label{thm:GMP}
Suppose $p(z_0)\ne0$ and $k=\min\{j\geq1:p^{(j)}(z_0)\ne0\}$. Write $\overline{p(z_0)}p^{(k)}(z_0)=r_0e^{i\varphi}$, where $r_0>0$. Along a fixed direction $e^{i\theta}$, the modulus initially increases if $\cos(k\theta+\varphi)>0$ and decreases if $\cos(k\theta+\varphi)<0$. These conditions define $2k$ alternating open sectors, each of angle $\pi/k$. The bounding rays, where $\cos(k\theta+\varphi)=0$, are excluded; the initial change along those rays depends on higher-order terms.
\end{theorem}
\begin{proof}
Taylor expansion gives, as $t\downarrow0$,
\[
 |p(z_0+te^{i\theta})|^2-|p(z_0)|^2
 =\frac{2r_0}{k!}t^k\cos(k\theta+\varphi)+O(t^{k+1}).
\]
The nonzero leading coefficient determines the sign. The zeros of the cosine give the bounding rays, on which higher-order terms must be examined separately.
\end{proof}
For $k=1$, the Newton direction is a positive multiple of the negative gradient direction in \eqref{eq:gradient}. For $k>1$, the GMP gives several descent sectors even though the gradient vanishes.

\subsection{The continuous map and the escape map}
For $p(z)\ne0$, define
\begin{equation}\label{eq:u}
\begin{aligned}
 k=k(z)&=\min\{j\geq1:a_j(z)\ne0\},\qquad
 u=p(z)\overline{a_k(z)},\\
 \gamma_k&=2\Rea(u^{k-1}),\qquad
 \delta_k=-2\Ima(u^{k-1}).
\end{aligned}
\end{equation}
and $c_k=\max\{|\gamma_k|,|\delta_k|\}>0$. Choose an angle by the deterministic rule
\begin{equation}\label{eq:angle}
 \vartheta=
 \begin{cases}
 0,& |\gamma_k|\geq|\delta_k|,\ \gamma_k<0,\\
 \pi/k,& |\gamma_k|\geq|\delta_k|,\ \gamma_k>0,\\
 \pi/(2k),& |\delta_k|>|\gamma_k|,\ \delta_k<0,\\
 3\pi/(2k),& |\delta_k|>|\gamma_k|,\ \delta_k>0.
 \end{cases}
\end{equation}
The selected nonzero component always has a sign, so the cases are exhaustive. Put
\begin{equation}\label{eq:escape}
 C(z)=\frac{c_k|u|^{2-k}}{18A(z)^2},\qquad
 E(z)=z+C(z)\frac{u}{|u|}e^{i\vartheta}
 \quad (p(z)\ne0),
\end{equation}
and set $E(z)=z$ when $p(z)=0$. We call $E$ the \emph{critical-point escape modification} of RNM.

\begin{proposition}\label{prop:step}
At a noncritical nonroot point, $E(z)=F(z)$ and
\begin{equation}\label{eq:damped}
 F(z)=z-\tau(z)\frac{p(z)}{p'(z)},\qquad
 \tau(z)=\frac{|p'(z)|^2}{9A(z)^2}\in(0,1/9].
\end{equation}
Both maps have step length at most $1/9$. The map $F$ is continuous and locally Lipschitz, and its fixed points are exactly the zeros of $pp'$.
\end{proposition}
\begin{proof}
For $k=1$, $c_k=2$, $\vartheta=\pi$, and $u=p\overline{p'}$. Substitution gives \eqref{eq:damped}. In general,
\begin{equation}\label{eq:Cbound}
 |u|^{k-1}\leq c_k\leq2|u|^{k-1},\qquad
 \frac{|u|}{18A(z)^2}\leq C\leq\frac{|u|}{9A(z)^2}\leq\frac19,
\end{equation}
because $|u|=|a_0(z)a_k(z)|\leq A(z)^2$. For $F$, the same length bound follows from $|p(z)|,|p'(z)|\leq A(z)$. A maximum of finitely many locally Lipschitz functions is locally Lipschitz. Since $A(z)\geq|a_n|>0$, this proves the asserted regularity of $F$. Its fixed-point characterization follows from \eqref{eq:F}.
\end{proof}

\begin{lemma}[Phase identity]\label{lem:phase}
With the quantities above and any $s>0$,
\[
 \overline u(sue^{i\vartheta})^k
 +u\overline{(sue^{i\vartheta})}^{\,k}
 =-c_k|u|^2s^k.
\]
\end{lemma}
\begin{proof}
The left-hand side equals $s^k|u|^2\bigl(\gamma_k\cos(k\vartheta)+\delta_k\sin(k\vartheta)\bigr)$. Each case in \eqref{eq:angle} makes the parenthesized quantity $-c_k$.
\end{proof}

\begin{theorem}[Modulus reduction]\label{thm:reduction}
If $p(z)\ne0$, then the escape update satisfies
\begin{align}\label{eq:reduction}
 |p(E(z))|^2-|p(z)|^2
 &\leq-9A(z)^2C(z)^{k+1}\notag\\
 &\leq-\frac{|p(z)p^{(k)}(z)|^{k+1}}
 {2\,18^k(k!)^{k+1}A(z)^{2k}}<0.
\end{align}
For the continuous map, at every $z\in\C$,
\begin{equation}\label{eq:descentF}
 |p(F(z))|^2\leq|p(z)|^2-
 \frac{|p(z)p'(z)|^2}{9A(z)^2}.
\end{equation}
Thus $F$ strictly decreases the modulus except at its fixed points, whereas $E$ strictly decreases it at every nonroot point.
\end{theorem}
\begin{proof}
Fix $z$, write $a_j=a_j(z)$ and $C=C(z)$, and put $w=C(u/|u|)e^{i\vartheta}$. Taylor's formula gives
\[
 p(z+w)=a_0+a_kw^k+W,\qquad W=\sum_{j=k+1}^na_jw^j.
\]
By Lemma~\ref{lem:phase}, $2\Rea(\overline{a_0}a_kw^k)=-c_k|u|^{2-k}C^k$. Using $|w|=C\leq1/9$, we have $|W|\leq A(z)C^{k+1}/(1-C)$. The remaining terms in the squared-modulus expansion are bounded by
\begin{align*}
 2|a_0W|+|a_kw^k|^2+2|a_kw^kW|+|W|^2
 &\leq A(z)^2\left(\frac{2C^{k+1}}{1-C}+C^{2k}
   +\frac{2C^{2k+1}}{1-C}+\frac{C^{2k+2}}{(1-C)^2}\right)\\
 &\leq\frac{6A(z)^2C^{k+1}}{1-C}.
\end{align*}
For the last inequality use $k\geq1$, $C\leq1/9$, and $C^{k+1}\leq1-C$. Since $c_k|u|^{2-k}=18A(z)^2C$ and $(1-C)^{-1}\leq3/2$, the full change is at most
\[
 -18A(z)^2C^{k+1}+\frac{6A(z)^2C^{k+1}}{1-C}
 \leq-9A(z)^2C^{k+1}.
\]
Substitution of the lower bound for $C$ in \eqref{eq:Cbound}, with $|u|=|p(z)p^{(k)}(z)|/k!$, proves the second inequality in \eqref{eq:reduction}.

At a noncritical nonroot point $k=1$ and $C=|p(z)p'(z)|/(9A(z)^2)$, giving \eqref{eq:descentF}. At a root or a critical point, $F(z)=z$ and the right-hand decrement is zero.
\end{proof}

\begin{remark}\label{rem:discontinuity}
At a nonroot critical point $c$, Theorem~\ref{thm:reduction} implies $E(c)\ne c$, while $F(z)\to c$ as $z\to c$. Thus $E$ is discontinuous there. For $p(z)=z^2-1$, $F(0)=0$ but $E(0)=-1/9$. The orbit of the latter point approaches $-1$, whereas positive real seeds arbitrarily close to zero approach $+1$. In particular, the $E$-root basin $\mathcal A_E(\theta)=\{z:E^{\circ t}(z)\to\theta\}$ need not be open. This distinction will be maintained throughout the basin analysis.
\end{remark}

\section{Global convergence and root basins}\label{sec:convergence}
Fix an initial seed $z_0\in\C$. Write $z_t$ for the point after $t$ updates, so that $z_{t+1}=F(z_t)$ for the continuous RNM map, or $z_{t+1}=E(z_t)$ for its escape modification. Define
\begin{equation}\label{eq:sublevel}
 g(z)=|p(z)p'(z)|,\quad S_0=\{z:|p(z)|\leq|p(z_0)|\},\quad
 A_0=\max_{z\in S_0}A(z).
\end{equation}
The set $S_0$ is compact because $|p(z)|\to\infty$ as $|z|\to\infty$.

\begin{lemma}\label{lem:finitecluster}
A bounded sequence in the Euclidean plane converges if its successive differences tend to zero and its set of accumulation points is finite.
\end{lemma}
\begin{proof}
Surround the finitely many accumulation points by disjoint closed balls of sufficiently small radius, separated by positive distances. Eventually all terms lie in their union; otherwise a bounded subsequence outside that union would have another accumulation point. Eventually a successive difference is shorter than every separation between the balls, so all sufficiently late terms stay in one ball. Repeating with smaller radii proves convergence to one accumulation point.
\end{proof}

\begin{theorem}[Global convergence and stationarity bound]\label{thm:global}
Every orbit of $F$ and every orbit of $E$ converges to a root or a critical point of $p$. For either map, if $t_\varepsilon$ is the first index with $g(z_{t_\varepsilon})\leq\varepsilon$, where $\varepsilon>0$, then
\begin{equation}\label{eq:complexity}
 t_\varepsilon\leq\frac{9A_0^2|p(z_0)|^2}{\varepsilon^2}.
\end{equation}
Moreover,
\begin{equation}\label{eq:gsum}
 \sum_{t=0}^{\infty}g(z_t)^2\leq9A_0^2|p(z_0)|^2.
\end{equation}
\end{theorem}
\begin{proof}
Theorem~\ref{thm:reduction} keeps the orbit in $S_0$. For $E$, inequality \eqref{eq:descentF} also holds: at a nonroot critical point its right-hand decrement is zero and the actual change is negative. Summing this inequality and using $A(z_t)\leq A_0$ gives \eqref{eq:gsum}. Hence $g(z_t)\to0$ and every accumulation point belongs to the finite set of zeros of $pp'$.

For $F$,
\[
 |z_{t+1}-z_t|=\frac{g(z_t)}{9A(z_t)^2}
 \leq\frac{g(z_t)}{9|a_n|^2}\longrightarrow0.
\]
Lemma~\ref{lem:finitecluster} proves convergence. An $E$-orbit can visit an exact nonroot critical point at most once, because the modulus strictly decreases at all its nonroot steps. There are finitely many such points. After its last such visit, the orbit follows $F$, so it also converges. Roots, once reached, remain fixed.

Before index $t_\varepsilon$, each squared-modulus decrease exceeds $\varepsilon^2/(9A_0^2)$. Telescoping and nonnegativity of the squared modulus yield \eqref{eq:complexity}.
\end{proof}

\paragraph{A coefficient bound for the descent constant.}
The constant $A_0$ in \eqref{eq:complexity} has an explicit upper bound. Put $r_0=|p(z_0)|$ and
\[
 R_0=1+\frac{\max\{|a_0|+r_0,|a_1|,\ldots,|a_{n-1}|\}}{|a_n|},
 \qquad
 \widehat A_0=\max_{0\leq j\leq n}\sum_{k=j}^n\binom{k}{j}|a_k|R_0^{k-j}.
\]
For $z\in S_0$, apply the Cauchy root bound to $p(\zeta)-p(z)$ to obtain $|z|\leq R_0$. The derivative formula then gives $A_0\leq\widehat A_0$. Thus the iteration bound can be evaluated from the coefficients and the seed without maximizing over $S_0$.

\begin{definition}\label{def:M}
Let $p$ be a polynomial of degree at least one, and let $\Crit(p)=\{c:p'(c)=0,\ p(c)\ne0\}$. When $\Crit(p)$ is nonempty, the \emph{nonroot critical
threshold} is $M=\min_{c\in\Crit(p)}|p(c)|>0$. \end{definition}
Statements involving $M$ assume $\Crit(p)\ne\varnothing$. If $\Crit(p)=\varnothing$, Theorem~\ref{thm:global} already gives root convergence for every seed.

\begin{corollary}[Root convergence below the critical threshold]\label{cor:threshold}
An $F$-orbit starting with $|p(z_0)|<M$ converges to a root. The same conclusion holds at equality if $z_0$ is not a nonroot critical point. An $E$-orbit starting with $|p(z_0)|\leq M$ converges to a root.
\end{corollary}
\begin{proof}
No nonroot critical point belongs to a strict sublevel below $M$. Apply Theorem~\ref{thm:global}. At equality, the stated nonroot initial points undergo a strict first decrease and enter such a sublevel. A root seed already gives a constant root orbit.
\end{proof}

In particular, for a cubic one can solve the quadratic equation $p'(z)=0$ and select a critical point with smallest modulus value. If it is a root, stop. Otherwise the escape update moves below $M$, and the subsequent orbit converges to a root. This special seed construction uses the low degree of $p'$; it is not a claim that $M$ is cheaply available for a general polynomial. Section~\ref{subsec:cubic} gives an explicit certification and iteration bound for this cubic procedure.

The bound \eqref{eq:complexity} is a stationarity bound. The product $g$ can be small near a nonroot critical point, and its smallness alone does not certify a root. Its exponent in $\varepsilon$ is independent of degree, but its constants and the work per iteration are not. Section~\ref{sec:implementation} makes the arithmetic cost explicit.

\begin{theorem}[Local attraction and openness]\label{thm:open}
For each root $\theta$, the $F$-basin $\mathcal A(\theta)=\{z:F^{\circ t}(z)\to\theta\}$ is an open subset of $\C$. At a simple root,
\begin{equation}\label{eq:linearroot}
 F(\theta+w)-\theta=q_\theta w+O(|w|^2),\qquad
 q_\theta=1-\frac{|p'(\theta)|^2}{9A(\theta)^2}\in[8/9,1).
\end{equation}
At a root of multiplicity $m\geq2$, a sufficiently small disk is also attracted to the root.
\end{theorem}
\begin{proof}
Local Lipschitz continuity of $A(z)$ and Taylor expansion give \eqref{eq:linearroot}, which implies contraction toward the root on a sufficiently small disk. For multiplicity $m\geq2$, write $p(\theta+w)=a w^m+O(w^{m+1})$, $a\ne0$. Then, with $b=m|a|^2/(9A(\theta)^2)>0$,
\begin{equation}\label{eq:multipleroot}
 |F(\theta+w)-\theta|^2
 =|w|^2-2b|w|^{2m}+O(|w|^{2m+1}).
\end{equation}
Choose a small disk on which the change is at most $-b|w|^{2m}$. Radii decrease and cannot have a positive limit, proving local attraction. Let $D$ be one of these attracting disks. Every point in the basin eventually enters $D$, so $\mathcal A(\theta)=\bigcup_{j\geq0}(F^{\circ j})^{-1}(D)$. Continuity of $F$ proves openness.
\end{proof}

\begin{proposition}\label{prop:no-local-critical-disc}
For any $c\in\Crit(p)$, no disk centered at $c$ is contained in its $F$-basin $\mathcal B(c)=\{z:F^{\circ t}(z)\to c\}$.
\end{proposition}
\begin{proof}
The GMP gives a point $w$ with $|p(w)|<|p(c)|$ in every such disk. Modulus descent prevents the orbit of $w$ from approaching $c$.
\end{proof}
This local observation alone does not establish empty interior of the entire critical basin. A global result will be proved in Section~\ref{sec:geometry}.

\section{Periodically safeguarded Newton acceleration}\label{sec:hybrid}
Fix a positive integer $T$. Between RNM updates, accept a Newton step whenever it is defined and strictly decreases $|p|$. Require an RNM update after at most $T-1$ consecutive accepted Newton steps. Here the RNM update is the escape map $E$, so an exact nonroot critical point is also handled.

More explicitly, let $d_t\in\{0,\ldots,T-1\}$ count consecutive accepted Newton steps since the most recent RNM update, and set $d_0=0$. A Newton trial at a nonroot point $z$ succeeds when $p'(z)\ne0$ and $|p(N(z))|<|p(z)|$. The algorithm $H_T$ updates its state by
\begin{equation}\label{eq:hybrid}
 (z_{t+1},d_{t+1})=
 \begin{cases}
 (z_t,0),&p(z_t)=0,\\
 (E(z_t),0),&p(z_t)\ne0,\ d_t=T-1,\\
 (N(z_t),d_t+1),&p(z_t)\ne0,\ d_t<T-1,\ \text{trial succeeds},\\
 (E(z_t),0),&\text{otherwise}.
 \end{cases}
\end{equation}
The trial is evaluated only when an RNM update is not already due. Thus $T=1$ gives pure $E$ iteration; for larger $T$, there is at least one $E$-update in every $T$ consecutive updates. The algorithm acts on $(z,d)$, rather than being an autonomous map of $z$ alone. An implementation can stop when it finds an exact root; holding that root fixed is convenient for the convergence statements.

\begin{theorem}[Periodic safeguard: convergence and complexity]\label{thm:hybrid}
For every initial seed and positive integer $T$, the iterates of $H_T$ converge to a root or a critical point of $p$. They converge to a root if $|p(z_0)|\leq M$. Let $S_0,A_0,g$ be as in \eqref{eq:sublevel}. The first index $t_\varepsilon$ with $g(z_{t_\varepsilon})\leq\varepsilon$, for $\varepsilon>0$, satisfies
\begin{equation}\label{eq:hybridcomplexity}
 t_\varepsilon\leq
 T\left\lceil\frac{9A_0^2|p(z_0)|^2}{\varepsilon^2}\right\rceil.
\end{equation}
The same conclusions hold for any rule for choosing updates that uses only strictly modulus-decreasing Newton steps and $E$-updates, with at least one $E$-update in every $T$ consecutive updates.
\end{theorem}
\begin{proof}
Every nonroot update strictly decreases the modulus, so all iterates lie in the compact set $S_0$. If a root is reached, the convergence assertion is immediate. Otherwise enumerate the $E$-update indices as $r_0<r_1<\cdots$, where $r_0\leq T-1$ and $r_{j+1}-r_j\leq T$.
At these indices Theorem~\ref{thm:reduction} gives
\[
 |p(z_{r_j+1})|^2\leq |p(z_{r_j})|^2-
 \frac{g(z_{r_j})^2}{9A_0^2}.
\]
All other updates decrease the modulus as well. Telescoping therefore yields
\begin{equation}\label{eq:periodicgsum}
 \sum_{j=0}^{\infty}g(z_{r_j})^2\leq9A_0^2|p(z_0)|^2.
\end{equation}
In particular, every accumulation point of $(z_{r_j})$ belongs to the finite set $Z=\{z:p(z)p'(z)=0\}$.

An exact nonroot critical point can be visited at most once, since its modulus strictly decreases when the orbit leaves it and never increases afterwards. There are finitely many such points. After the last such visit, every $E$-update equals an $F$-update.

We next propagate a limit across a block of at most $T$ updates. Suppose $z_{r_{j_k}}\to c\in Z$. If $c$ is a root, both $F$ and the Newton map, extended by fixing roots, are continuous at $c$. Indeed, for a root of multiplicity $m$, the quotient $p(z)/p'(z)$ has a removable zero at $c$, with expansion $(z-c)/m+O(|z-c|^2)$. If $c$ is a nonroot critical point, a Newton step starting at points tending to $c$ would have length tending to infinity. Such a step cannot be accepted, since both its endpoints would have to lie in $S_0$. Consequently all sufficiently late updates from points approaching this $c$ use $F$, and their endpoints also approach $c$. Induction on the offset gives
\begin{equation}\label{eq:blocklimit}
 z_{r_{j_k}+\ell}\longrightarrow c\qquad(0\leq\ell\leq T).
\end{equation}

For an arbitrary sequence of sufficiently large indices, choose the preceding $r_j$. The offset is in $\{0,\ldots,T-1\}$. Pass to a subsequence on which it is constant and the starting points $z_{r_j}$ converge; their limit lies in $Z$. Equation~\eqref{eq:blocklimit} shows that the chosen iterates and their successors have this same limit. It follows that all accumulation points of the full orbit lie in $Z$ and that $|z_{t+1}-z_t|\to0$. Lemma~\ref{lem:finitecluster} proves convergence of the entire orbit.

For the iteration bound, put $C_\varepsilon=9A_0^2|p(z_0)|^2/\varepsilon^2$. If $g(z_0)\leq\varepsilon$, the assertion is immediate. Otherwise, before the first successful stationarity test, every $E$-update decreases the squared modulus by more than $\varepsilon^2/(9A_0^2)$. Thus $\lceil C_\varepsilon\rceil$ such updates without a successful test would contradict nonnegativity of the squared modulus. Every $T$ updates include one $E$-update, proving \eqref{eq:hybridcomplexity}. Finally, if a nonroot seed has modulus at most $M$, its first update enters a strict sublevel below $M$. A nonroot critical limit is then impossible. The proof used only the stated requirement on the frequency of RNM updates, so it proves the final assertion as well.
\end{proof}

The bound \eqref{eq:hybridcomplexity} is $O(T/\varepsilon^2)$ for reaching stationarity, not for achieving absolute root error $\varepsilon$. The parameter $T$ controls how many Newton trials may occur between RNM updates. An unsuccessful trial immediately triggers $E$ and resets the counter.

Keeping the periodic requirement forever does not, in general, give quadratic convergence per update, even at a simple root: the infinitely many RNM updates retain the nonzero linear factor in \eqref{eq:linearroot}. The next section justifies removing that requirement permanently after a successful pointwise Newton certificate. No identity between the basins of $H_T$ and those of $F$ is asserted.

\section{Certified switching to Newton's method}\label{sec:certification}
This section explains how RNM prepares an iterate for a
certified switch to Newton's quadratic convergence.
Section~5.1 recalls Smale's pointwise $\alpha$-test and
gives a neighborhood of each simple root where it succeeds.
Section~5.2 proves eventual certification for orbits
approaching simple roots and bounds the preparatory phase
under additional assumptions.
Section~5.3 gives a deterministic cubic solver with an
explicit bound for reaching certification.
Finally, Section~5.4 introduces RNM-guided Newton for
general polynomials, proves that it retains the cubic
bound, and presents brief computational comparisons.
\subsection{Pointwise certification and a local root-dependent neighborhood}
At a point with $p'(z)\ne0$, define
\begin{equation}\label{eq:smale}
 \beta_p(z)=\left|\frac{p(z)}{p'(z)}\right|,\quad
 \gamma_p(z)=\max_{2\leq j\leq n}\left|\frac{p^{(j)}(z)}{j!p'(z)}\right|^{1/(j-1)},\quad
 \alpha_p(z)=\beta_p(z)\gamma_p(z).
\end{equation}
We use the following standard form of point estimation \cite{Smale1986,AlphaCertified}. For a broader treatment, including Newton estimates and approximate zeros, see Petkovi\'{c}'s monograph \cite[Chapter~2]{Petkovic2008}. The index $s$ below counts Newton steps; $T$ continues to denote the maximum gap between RNM updates before certification.

\begin{theorem}[Pointwise $\alpha$-certification]\label{thm:alpha}
If $\alpha_p(w)<\alpha_*:=({13-3\sqrt{17}})/4$
then there is a root $\xi$ such that all Newton iterates starting at $w$ are defined, with roots held fixed, and
\begin{equation}\label{eq:alphaerror}
 |N^{\circ s}(w)-\xi|\leq2^{1-2^s}|w-\xi|,
 \qquad |w-\xi|\leq2\beta_p(w),\qquad s\geq0.
\end{equation}
\end{theorem}
Use the conservative threshold $\eta=1/10<\alpha_*$. The test must be evaluated at the current iterate:
\begin{equation}\label{eq:pointthreshold}
 \alpha_p(w)\leq\eta\quad\Longleftrightarrow\quad
 |p(w)p'(w)|\leq\frac{\eta|p'(w)|^2}{\gamma_p(w)}.
\end{equation}
The denominator is positive because $n\geq2$ and the highest derivative is nonzero. A threshold evaluated at a different point, such as a root, does not preserve this equivalence.

The following local statement makes precise the role of normalized derivatives at a simple root. It does not bound the time taken by a global orbit to enter the neighborhood.

\begin{lemma}[A certified neighborhood of a simple root]\label{lem:rootradius}
Let $\theta$ be a simple root and put $\gamma_\theta=\gamma_p(\theta)$ and $r_\theta=1/(20\gamma_\theta)$. For every $w$ with $|w-\theta|\leq r_\theta$, one has $p'(w)\ne0$ and $\alpha_p(w)<\eta$. Newton iteration from $w$ stays in this disk and converges quadratically to $\theta$, with
\begin{equation}\label{eq:localnewton}
 |N^{\circ s}(w)-\theta|\leq
 \frac{(2\gamma_\theta|w-\theta|)^{2^s}}{2\gamma_\theta},
 \qquad s\geq0.
\end{equation}
Only the root $\theta$ is required to be simple.
\end{lemma}
\begin{proof}
Write $p(\theta+v)=\sum_{j=1}^n a_j(\theta)v^j$, so that $|a_j(\theta)/a_1(\theta)|\leq\gamma_\theta^{j-1}$ for $j\geq2$. Set $u=\gamma_\theta|w-\theta|\leq1/20$ and $\psi(u)=1-4u+2u^2$. Summing the corresponding geometric series gives
\[
 \frac{|p'(w)|}{|p'(\theta)|}
 \geq1-\sum_{j=2}^{\infty}j u^{j-1}
 =\frac{\psi(u)}{(1-u)^2}>0,
 \qquad
 \frac{|p(w)|}{|p'(\theta)|}\leq\frac{|w-\theta|}{1-u}.
\]
Taylor shifting also gives, for $j\geq2$,
\[
 \frac{|p^{(j)}(w)|}{j!|p'(\theta)|}
 \leq\sum_{k=j}^{\infty}\binom{k}{j}\gamma_\theta^{k-1}|w-\theta|^{k-j}
 =\frac{\gamma_\theta^{j-1}}{(1-u)^{j+1}}.
\]
Since $0<\psi(u)\leq1$, division by the derivative lower bound yields
\[
 \beta_p(w)\leq\frac{|w-\theta|(1-u)}{\psi(u)},\qquad
 \gamma_p(w)\leq\frac{\gamma_\theta}{(1-u)\psi(u)},\qquad
 \alpha_p(w)\leq\frac{u}{\psi(u)^2}.
\]
The last bound is at most $(1/20)/(161/200)^2=2000/25921<1/10$. Finally, the Taylor series for $(w-\theta)p'(w)-p(w)$ gives
\[
 |N(w)-\theta|\leq
 \frac{\gamma_\theta|w-\theta|^2}{\psi(u)}
 \leq2\gamma_\theta|w-\theta|^2.
\]
Because $2\gamma_\theta r_\theta=1/10$, this estimate keeps the disk invariant. Iterating it proves \eqref{eq:localnewton} and identifies the Newton limit as $\theta$.
\end{proof}

\subsection{The preparatory phase and the Newton phase}
\begin{theorem}[Certified permanent Newton transition]\label{thm:switch}
Run $F$, $E$, or the periodic hybrid $H_T$, holding roots fixed. At each nonroot iterate $w$, switch permanently to standard Newton iteration if $p'(w)\ne0$ and $\alpha_p(w)\leq\eta$. For $H_T$, certification may instead be checked only immediately before an $E$-update would be made; all the conclusions below hold with this choice as well. A successful test removes the periodic RNM requirement. Then:
\begin{enumerate}[label=(\roman*)]
\item If the unswitched orbit converges to a simple root $\theta$, the test succeeds after finitely many preparatory updates, unless an exact root has already been found. This preparatory count depends on the polynomial and seed (and on $T$ for $H_T$), but not on the requested final accuracy. Any successful test guarantees convergence to a certified root $\xi$. A switch made sufficiently late in the unswitched orbit has $\xi=\theta$; an earlier certified switch need not be assumed to select that same root.
\item Suppose all roots of $p$ are simple and $|p(z_0)|<M$. Define
\begin{equation}\label{eq:qS}
 q_0=\min_{z\in S_0}\frac{\eta|p'(z)|^2}{\gamma_p(z)}>0,
 \qquad
 B=\begin{cases}1,&\text{for $F$ or $E$},\\T,&\text{for $H_T$}.\end{cases}
\end{equation}
The number of preparatory updates before a successful test or an exact root is at most
\begin{equation}\label{eq:switchbound}
 t_{\mathrm{cert}}\leq
 B\left\lceil\frac{9A_0^2|p(z_0)|^2}{q_0^2}\right\rceil.
\end{equation}
\item After a successful test at $w$, the bound
\begin{equation}\label{eq:certstop}
 |N^{\circ s}(w)-\xi|\leq2^{2-2^s}\beta_p(w)
\end{equation}
certifies absolute root error without knowing $\xi$. The number of additional Newton steps for error at most $\varepsilon$ is $O\!\left(1+\ln\ln\bigl(e+\beta_p(w)/\varepsilon\bigr)\right)$. \end{enumerate}
\end{theorem}
\begin{proof}
For (i), convergence to the simple root places every sufficiently late unswitched iterate in the disk of Lemma~\ref{lem:rootradius}. This proves eventual certification and the assertion about a sufficiently late switch to the same root. The disk and its entry time do not depend on the final error tolerance. A test succeeding earlier still has the guarantee of Theorem~\ref{thm:alpha}, with its own associated root.

For (ii), $S_0$ contains no critical point: all its roots are simple and its modulus level is below every nonroot critical value. Hence $p'$ is nonzero on $S_0$, and $\gamma_p$ is continuous and strictly positive there. Compactness proves $q_0>0$. If $g(w)\leq q_0$ at a nonroot iterate, \eqref{eq:pointthreshold} certifies the test. For $F$ and $E$, apply Theorem~\ref{thm:global} with tolerance $q_0$. For $H_T$, before certification every checked $E$-update point has $g>q_0$. Each such update decreases $|p|^2$ by more than $q_0^2/(9A_0^2)$, and successive checks are separated by at most $T$ updates. Telescoping proves \eqref{eq:switchbound} even with checks restricted to those indices; earlier checks can only improve the bound. Finally, combining the two inequalities in \eqref{eq:alphaerror} proves \eqref{eq:certstop}; solving it for $s$ proves (iii).
\end{proof}

In particular, for an orbit converging to a simple root there is a finite preparatory index $K$, independent of $\varepsilon$, for which a certified Newton tail achieves the total count
\begin{equation}\label{eq:totalcomplexity}
 K+O\!\left(1+\ln\ln\bigl(e+\beta_p(z_K)/\varepsilon\bigr)\right).
\end{equation}
For $H_T$, $K$ may also depend on $T$. Only the limiting root must be simple for this conclusion; other roots may be multiple. At a root of multiplicity $m>1$, standard Newton instead has $N(\theta+v)-\theta=(1-1/m)v+O(|v|^2)$, so quadratic convergence is not generally available without a multiplicity-aware modification.

The constant $q_0$ depends on the entire initial sublevel set, not only on derivatives at a root. It is used for analysis; the algorithm only computes the pointwise test at its actual iterates. Neither $M$ nor $q_0$ needs to be computed to run the method. If an escape or hybrid seed starts at modulus $M$, first take its strict nonroot update and then apply part~(ii) to the resulting sublevel set, provided all roots are simple.

Checking certification only before an $E$-update reuses the normalized derivatives already needed there and leaves intermediate Newton trials inexpensive.

\begin{remark}[Dependence on the initial seed]
The preparatory count is independent of the requested
accuracy but need not be uniform over initial seeds.
For $p(z)=z^2-1$, every RNM orbit starting at
$x_0\in(0,1)$ converges to the simple root $1$. However,
for $0<x<1/4$,
\[
 F(x)\leq\frac{11}{9}x,
 \qquad
 \alpha_p(x)=\frac{1-x^2}{4x^2}>\frac{1}{10}.
\]
Thus, for $0<x_0<1/4$, reaching a point satisfying
$\alpha_p\leq1/10$ requires at least $\ln(1/(4x_0))/\ln(11/9)$ RNM updates. This lower bound tends to infinity as
$x_0\downarrow0$, although the polynomial and limiting
root remain fixed.
\end{remark}

\subsection{A certified algorithm for cubic equations}
\label{subsec:cubic}

Solving cubic equations is a classical algebraic problem with modern applications in computer graphics \cite{Blinn2006,Yuksel2022}. For their connections with polynomiography and modulus visualization, see \cite{CubicLASER2024}. The critical points of a cubic polynomial can be computed
by solving a quadratic equation. This permits a deterministic
initialization of RNM, already proposed in \cite{RNM2020}. The theorem below makes its transition to certified Newton iteration quantitative. Its purpose is an explicit convergence guarantee for a descent-based cubic solver, rather than a claim of superiority to algebraic formulas or specialized solvers. We use reduced cubics: division by the leading coefficient and translation of the variable reduce every cubic to this form and preserve the iteration counts and absolute root error.

\begin{theorem}[A certified RNM--Newton algorithm for reduced cubics]
\label{thm:cubic-certified}
Let $p(z)=z^3+az+b$, with $a,b\in\mathbb C$, and use $F$ from \eqref{eq:F} with $A(z)=\max\{|p(z)|,|3z^2+a|,3|z|,1\}$. Choose either square root $d=\sqrt{-a/3}$. The critical points are $\pm d$, with values $p(\pm d)=b\mp2d^3$; put $M=\min\{|b-2d^3|,|b+2d^3|\}$.

If $M=0$, one of the critical points is already a root.
Otherwise, choose $c\in\{d,-d\}$ with $|p(c)|=M$ and take
one escape step $w=E(c)$, so that $r=|p(w)|<M$. If $r=0$, stop. For $0<r<M$, define
\begin{align}
 R=1+\max\{|a|,\ |b|+r\}, \quad
 B=3R^2+|a|, \quad
 m=3\left(\frac{M-r}{2}\right)^{2/3},
 \label{eq:cubic-gap}
\end{align}
and
\begin{equation}\label{eq:cubic-entry-count}
 K=
 \left\lceil
 \frac{18B^2}{m^2}
 \max\left\{0,\ln\frac{6r}{M}\right\}
 \right\rceil,
\end{equation}
Then the following conclusions hold.
\begin{enumerate}[label=(\roman*)]
\item Every point satisfying $|p(z)|<M$ obeys
\begin{equation}\label{eq:cubic-alpha-bound}
 \alpha_p(z)\le
 \frac{|p(z)|}{2(M-|p(z)|)}.
\end{equation}
Consequently, $|p(z)|\le M/6$ implies
$\alpha_p(z)\le1/10$.

\item Starting at $z_0=w$, the RNM iteration
$z_{k+1}=F(z_k)$ reaches a point $v=z_\ell$ satisfying
$|p(v)|\le M/6$ after at most $K$ updates.
A permanent switch at $v$ to ordinary Newton iteration, $N(z)=(2z^3-b)/(3z^2+a)$, converges quadratically to a root $\theta$, with
\begin{equation}\label{eq:cubic-newton-error}
 |N^{\circ s}(v)-\theta|
 \le \frac{2M}{3m}\,2^{-2^s},
 \qquad s\ge0.
\end{equation}

\item For $\varepsilon>0$, put
\begin{equation}\label{eq:cubic-newton-count}
 s_\varepsilon=
 \left\lceil
 \log_2\left(
 \max\left\{
 1,\log_2\frac{2M}{3m\varepsilon}
 \right\}
 \right)
 \right\rceil.
\end{equation}
After the initial calculation of the critical points
and the displayed constants, at most $1+K+s_\varepsilon$ updates suffice to obtain absolute root error at most
$\varepsilon$: one escape step, at most $K$ RNM steps,
and at most $s_\varepsilon$ Newton steps.
\end{enumerate}

The theorem includes $a=0$, when the critical points
coincide. The counts concern exact-arithmetic iteration
updates. For a fixed polynomial and the stated
initialization, the preparatory bound is independent of
$\varepsilon$, and the additional Newton count is $O(\ln\ln(1/\varepsilon))$ as $\varepsilon\downarrow0$. \end{theorem}

\begin{proof}
Since $a=-3d^2$, $p'(z)=3(z-d)(z+d)$ and $p(\pm d)=b\mp2d^3$. A multiple root must also be a critical point.
Thus, if $M=0$, the initial calculation finds a root;
otherwise all roots are simple.
The modulus-reduction theorem gives $|p(E(c))|<|p(c)|=M$, which justifies the initialization.

We first prove the certification estimate and a lower
bound for $|p'|$. Suppose $|p(z)|<M$.
Relabel the critical points as $c_1,c_2$ so that $h=z-c_1$, $g=z-c_2$, and $|h|\le|g|$. Neither $h$ nor $g$ is zero. Since $c_1+c_2=0$,
\begin{equation}\label{eq:cubic-identities}
 p'(z)=3hg,\qquad
 p''(z)=3(h+g),\qquad
 p(z)-p(c_1)=\frac12h^2(3g-h).
\end{equation}
These identities also hold when $d=0$.

Set $D=|p(z)-p(c_1)|$ and $q=h/g$. Then $|q|\le1$ and
\begin{equation}\label{eq:cubic-critical-gap}
 D\ge |p(c_1)|-|p(z)|
   \ge M-|p(z)|>0.
\end{equation}
The identities above imply
\[
 D\,\frac{|p''(z)|}{2|p'(z)|^2}
 =\frac{|(3-q)(1+q)|}{12}
 =\frac{|3+2q-q^2|}{12}
 \le\frac12
\]
and
\[
 D\,\frac{1}{|p'(z)|^{3/2}}
 =\frac{\sqrt{|q|}\,|3-q|}{6\sqrt3}
 \le\frac{2}{3\sqrt3}<\frac12.
\]
For this cubic,
\[
 \alpha_p(z)
 =
 |p(z)|\max\left\{
 \frac{|p''(z)|}{2|p'(z)|^2},
 \frac{1}{|p'(z)|^{3/2}}
 \right\}.
\]
Consequently,
\[
 \alpha_p(z)\le\frac{|p(z)|}{2D}
 \le\frac{|p(z)|}{2(M-|p(z)|)}.
\]
At $|p(z)|\le M/6$, the last expression is at most
$1/10$. This proves~(i).

The same identities also give
\[
 D\le2|h|^2|g|
   =\frac23|h|\,|p'(z)|.
\]
Because $|h|\le|g|$, $|h|^2\le|hg|=|p'(z)|/3$. It follows that
\[
 M-|p(z)|
 \le D
 \le\frac{2}{3\sqrt3}|p'(z)|^{3/2},
\]
and hence
\begin{equation}\label{eq:cubic-derivative-bound}
 |p'(z)|
 \ge
 3\left(\frac{M-|p(z)|}{2}\right)^{2/3}.
\end{equation}
In particular, $|p'(z)|\ge m$ on $S_r=\{z:|p(z)|\le r\}$.

We next bound $A(z)$ on $S_r$. Each $z\in S_r$ is a root of $\zeta^3+a\zeta+(b-p(z))$, so Cauchy's root bound gives $|z|\le1+\max\{|a|,|b-p(z)|\}\le R$. Since $R\ge1+r$, we have $|p(z)|\le r<B$, $|3z^2+a|\le B$, $3|z|\le3R\le B$, and $1\le B$. Thus $A(z)\le B$ throughout $S_r$.

RNM preserves $S_r$. Applying its modulus-reduction
estimate at each preparatory iterate yields
\[
 |p(z_{k+1})|^2
 \le
 |p(z_k)|^2
 \left(1-\frac{|p'(z_k)|^2}{9A(z_k)^2}\right) 
 \le
 |p(z_k)|^2
 \left(1-\frac{m^2}{9B^2}\right).
\]
Here $0<m\le B$. Iterating and using $1-x\le e^{-x}$
gives
\begin{equation}\label{eq:cubic-geometric-decay}
 |p(z_k)|
 \le
 r\exp\left(-\frac{m^2k}{18B^2}\right).
\end{equation}
If $r\le M/6$, then $K=0$ and no preparatory RNM
step is needed. Otherwise, the definition of $K$
gives $|p(z_K)|\le M/6$.
Thus the first such index $\ell$ satisfies $\ell\le K$.

At $v=z_\ell$, part~(i) gives $\alpha_p(v)\le1/10<\alpha_*$. Theorem~\ref{thm:alpha} therefore supplies a root $\theta$ with $|N^{\circ s}(v)-\theta|\le2^{2-2^s}\beta_p(v)$. Since $v\in S_r$, we have $\beta_p(v)\le M/(6m)$, giving \eqref{eq:cubic-newton-error} and proving~(ii).

Finally, $2^{s_\varepsilon}\ge\log_2(2M/(3m\varepsilon))$ by definition, so \eqref{eq:cubic-newton-error} yields $|N^{\circ s_\varepsilon}(v)-\theta|\le\varepsilon$. Adding the initial escape, at most $K$ RNM updates, and at most $s_\varepsilon$ Newton updates proves~(iii).
\end{proof}

The bound $K$ uses worst-case estimates on the entire sublevel set and can be very conservative. For $p(z)=z^3-1$, $c=0$, $M=1$, $E(c)=-e^{i\pi/3}/9$, and $r=728/729$ give $K=43{,}077{,}774$. Direct iteration reaches $|p|\leq1/6$ in $48$ RNM updates after the escape. Thus $K$ is an explicit certification bound, not an estimate of the observed running time.

\subsection{RNM-guided Newton}\label{subsec:rnm-guided}
The \emph{RNM-guided Newton method} applies to any nonconstant
polynomial. RNM-guided Newton aims to accelerate RNM by accepting
longer steps in the Newton direction whenever they achieve
RNM's guaranteed modulus decrease. It therefore preserves
the stationarity bound and the cubic preparatory bound,
while allowing faster progress in practice. Its additional
trial evaluations mean that a reduction in iteration count
does not necessarily imply a reduction in computing time.

With $A(z)$ as in \eqref{eq:A}, at a noncritical
nonroot point put $d_N=-p(z)/p'(z)$ and
$\tau(z)=|p'(z)|^2/(9A(z)^2)\in(0,1/9]$, so that
$F(z)=z+\tau(z)d_N$. Try $y=z+td_N$ for
$t=1,1/2,1/4,\ldots$ while $t>\tau(z)$, accepting the first
trial satisfying
\begin{equation}\label{eq:guided-test}
 |p(y)|^2\leq(1-\tau(z))|p(z)|^2.
\end{equation}
If none succeeds, take $F(z)$, which satisfies
\eqref{eq:guided-test} by \eqref{eq:descentF}. The trial list
is finite because $\tau(z)>0$. At an exact nonroot critical
point use $E(z)$, and stop at a root. This backtracking
line search \cite[Chapter~9]{BoydVandenberghe2004} uses RNM's
guaranteed decrease as its acceptance test.
Each trial requires evaluating $p(y)$. The estimate avoids a
separate evaluation at $F(z)$ when a trial succeeds; if the
fallback is used, its polynomial value is needed for the next
iteration. Thus trial evaluations remain part of the computational
cost, and this saving alone does not guarantee a speed advantage
over ordinary Newton.

Every accepted update, including a critical-point escape,
satisfies \eqref{eq:descentF} with the accepted point in place
of $F(z)$. Summing as in the proof of Theorem~\ref{thm:global}
gives $|p(z_k)p'(z_k)|\to0$ and the same
$O(\varepsilon^{-2})$ bound \eqref{eq:complexity} for reaching
$|p(z_k)p'(z_k)|\leq\varepsilon$.
The count concerns accepted updates; trial evaluations are
additional work. Every accumulation point is a root or critical
point, but this argument does not establish convergence of the
whole guided orbit to a single limit. The basin results below concern $F$.

A full Newton trial satisfying
$|p(N(z))|^2\leq(8/9)|p(z)|^2$ can be accepted without
computing $A(z)$, since $1-\tau(z)\geq8/9$.
Otherwise, compute $A(z)$, compare the already evaluated
Newton trial with \eqref{eq:guided-test}, and halve if needed.
The polynomial value at an accepted trial can be reused at the
next iteration.

\paragraph{Application to cubics.}
For the reduced cubic $p(z)=z^3+az+b$, the general formula
simplifies to $A(z)^2=\max\{|p(z)|^2,|p'(z)|^2,9|z|^2,1\}$.
The following corollary applies RNM-guided Newton to the
preparatory phase of Theorem~\ref{thm:cubic-certified}.

\begin{corollary}[RNM-guided Newton preserves the cubic bound]
\label{cor:cubic-guided}
Use the notation of Theorem~\ref{thm:cubic-certified}, with
$0<r<M$ and $z_0=w$. Replace its preparatory RNM updates by
RNM-guided Newton updates, then switch permanently to Newton
at the first $v$ satisfying $|p(v)|\leq M/6$.
This point is reached after at most $K$ accepted preparatory
updates. The certification and error bound of the theorem
remain valid, and $1+K+s_\varepsilon$ bounds the total number
of accepted updates, including the initial escape.
Trial evaluations are additional work.
\end{corollary}
\begin{proof}
The fallback satisfies \eqref{eq:guided-test} by
\eqref{eq:descentF}, and the trial list is finite because
$\tau(z)>0$. Every accepted update therefore obeys
\eqref{eq:guided-test} and remains in $S_r=\{z:|p(z)|\leq r\}$.
The cubic theorem gives $|p'(z)|\geq m$ and $A(z)\leq B$
on $S_r$, so
\[
 |p(z_{k+1})|^2
 \leq\left(1-\frac{|p'(z_k)|^2}{9A(z_k)^2}\right)|p(z_k)|^2
 \leq\left(1-\frac{m^2}{9B^2}\right)|p(z_k)|^2.
\]
Hence $|p(z_k)|\leq r\exp(-m^2k/(18B^2))$.
The same $K$ ensures entry into $|p|\leq M/6$, after which
the theorem's Newton-phase argument applies.
\end{proof}

\paragraph{Illustrative numerical comparisons.}
Table~\ref{tab:guided-timing} compares the time to compute one root
by ordinary Newton and RNM-guided Newton on two random sets of
reduced cubics. Both methods were timed on the same computer using the same
numerical precision. Both stop at a computed zero or, in the
Newton phase, when $|p(z)/p'(z)|\leq10^{-12}$.
The reported times summarize repeated runs.
The sets contain $256$ complex and $64$ real cubics,
and $512$ complex and $128$ real cubics, respectively. Complex
coefficient components are uniform in $[-2,2]$; real coefficients
are uniform in $[-3,3]$. Both methods start from the same escape
point $w$; the guided method switches permanently to Newton at $|p|\leq M/6$.
RNM-guided Newton took about $15$--$16\%$ less time on these samples. Mean update counts after the escape were $10.33$ versus $5.48$ and $10.39$ versus $5.45$ (Newton versus guided). These counts exclude rejected trials; the timings include their cost.

\begin{table}[htbp]
\centering
\caption{Time per root, averaged over each sample: medians of
$14$ repeated batch measurements, including initialization; parentheses give the $10$th--$90$th percentile range across batches.}
\label{tab:guided-timing}
\begin{tabular}{@{}lrr@{}}
\toprule
&\multicolumn{2}{c}{Time to compute one root ($\mu$s)}\\
\cmidrule(l){2-3}
Method & $320$ random cubics & $640$ random cubics\\
\midrule
Newton & $0.277\ (0.269\text{--}0.281)$ & $0.282\ (0.272\text{--}0.314)$\\
RNM-guided Newton & $0.235\ (0.229\text{--}0.240)$ & $0.237\ (0.230\text{--}0.251)$\\
\bottomrule
\end{tabular}

\end{table}

For $p(z)=z^3-1$, the near-critical seeds $z_0=1/9$, $10^{-2}$,
and $10^{-3}$ gave times of $0.304$, $0.553$, and $0.846\,\mu$s
for ordinary Newton, compared with $0.159$, $0.225$, and
$0.281\,\mu$s for RNM-guided Newton. All roots computed in these
cubic comparisons had absolute error at most $10^{-10}$ against
independently computed reference roots.

A further comparison used six degree-ten polynomials:
$z^{10}-1$, $z^{10}-z-1$, $z^{10}-5z^5+1$,
$z^{10}+(1+i)z^7-(2-i)z^3+1$, and two polynomials with random
real and complex coefficients. Each polynomial was tested from
$32$ ordinary complex seeds and six seeds near a nonroot critical
point. Both methods found a root in all $228$ cases, with observed
absolute errors below $10^{-12}$. For the near-critical seeds,
RNM-guided Newton reduced the average time per root by approximately
$25$--$91\%$, depending on the polynomial. For ordinary seeds it
was faster on three polynomials and slower on three. These samples
illustrate a benefit on the selected near-critical seeds and mixed behavior elsewhere. Near a critical point, backtracking can avoid the large steps of ordinary Newton; the comparison does not isolate an advantage of RNM's acceptance rule over other damped Newton methods. Evaluation counts are not reported, and no general speed advantage or improvement with degree is claimed.
\FloatBarrier

\section{Exact Voronoi basins for roots of unity}\label{sec:unity}
The next theorem concerns the continuous map $F$, including its fixed critical point at the origin. It also provides an example in which the geometry can be determined even when the nonroot critical point is not simple.

We will prove that, for $p(z)=z^n-1$, each root's basin under
$F$ is exactly its open Voronoi region, and all boundary
points converge to the origin. By rotational symmetry,
it suffices to analyze the sector containing the root $1$.
We show that this sector is preserved by $F$. The global
convergence theorem then leaves only two possible limits:
$1$ or the origin. A distance estimate excludes the origin
for points inside the sector, while points on its boundary
rays converge to the origin. Rotation gives the result
for the remaining sectors, which are precisely the open
Voronoi regions.

\begin{theorem}\label{thm:unity}
Let $p(z)=z^n-1$, $n\geq2$. Write its roots as $\omega_j=e^{i\theta_j}$, where $\theta_j=2\pi j/n$, $j=0,\ldots,n-1$, and $i^2=-1$. The $F$-basin of $\omega_j$ is exactly its open Euclidean Voronoi region,
\begin{equation}\label{eq:voronoi}
 \mathcal A(\omega_j)=\left\{z\ne0:\left|\arg(z/\omega_j)\right|<\frac\pi n\right\},
 \qquad j=0,\ldots,n-1,
\end{equation}
where the argument is chosen in $(-\pi,\pi]$. Every point on a boundary ray $\arg z=(2j+1)\pi/n$, together with the origin, converges to the critical point $0$.
\end{theorem}
\begin{proof}
We show that each open sector is invariant and that every orbit
in it converges to its central root. We then identify the sectors
with the Voronoi regions and describe their boundary orbits.

For $z=re^{i\theta}\ne0$, put $b=nr^{n-2}/(9A(z)^2)>0$. Then
\begin{equation}\label{eq:unitystep}
 F(z)=re^{i\theta}\bigl(1-br^n+b e^{-in\theta}\bigr).
\end{equation}
Denote the complex factor in parentheses by $Q(z)$, so
$F(z)=zQ(z)$. Since $A(z)\geq nr^{n-1}$ and
$A(z)\geq\max\{1,n(n-1)r^{n-2}/2\}$, we have
\begin{equation}\label{eq:unitybounds}
 br^n\leq\frac1{9n},\qquad b(n-1)\leq\frac29,\qquad
 b(r^n+n-1)<1.
\end{equation}
In particular, $\Rea Q(z)\geq1-b(r^n+1)>0$, so $F(z)\ne0$.

Rotations by roots of unity and complex conjugation commute with
$F$. It therefore suffices first to consider the upper half of the
sector about $1$, where $0<\theta<\pi/n$. Here
$\Ima Q(z)=-b\sin(n\theta)<0$, whereas
\[
 \Ima F(z)
 =r\bigl((1-br^n)\sin\theta-b\sin((n-1)\theta)\bigr)
 \geq r\bigl[1-b(r^n+n-1)\bigr]\sin\theta>0.
\]
Here we use $\sin((n-1)\theta)\leq(n-1)\sin\theta$,
which follows by induction from the sine addition formula
and $|\sin(m\theta)|\leq m|\sin\theta|$.
Thus $0<\arg F(z)<\theta$: the iterate stays in the upper half
and moves angularly toward the positive real axis. That axis is
also invariant. Reflection and rotation prove invariance of every
open sector in \eqref{eq:voronoi}.

By Theorem~\ref{thm:global}, an orbit in the sector about $1$
can converge only to $1$ or $0$. To exclude $0$, consider
$0\leq\theta<\pi/n$ and set $\phi=\pi/n-\theta$. The positive
quantity $d(z)=r\sin\phi$ is the distance to the upper boundary
ray. Applying this distance formula to \eqref{eq:unitystep} gives
\[
 \begin{aligned}
 d(F(z))&=(1-br^n)d(z)+br\sin((n-1)\phi)\\
 &\geq\bigl[1+b(1-r^n)\bigr]d(z),
 \end{aligned}
\]
because $\sin((n-1)\phi)\geq\sin\phi$ for
$0<\phi\leq\pi/n$. Indeed, their difference is
$2\cos(n\phi/2)\sin((n-2)\phi/2)\geq0$.
If an orbit tended to $0$, its radii would eventually be below
$1$, and its positive distances $d(z_t)$ would then strictly
increase. This contradicts $d(z_t)\leq|z_t|\to0$.
Conjugation handles the lower half of the sector. Hence every
point in this sector converges to $1$, and rotation gives the
corresponding conclusion for every root.

On a boundary ray, $e^{-in\theta}=-1$, so
$Q(z)=1-b(r^n+1)\in(0,1)$. The orbit stays on that ray and
moves inward. The ray contains no root; the global convergence
theorem therefore makes its limit $0$. The origin itself is fixed.

Finally, the identity
$|z-\omega_j|^2=|z|^2+1-2\Rea(z\overline{\omega_j})$
shows that $z$ is strictly closer to $\omega_j$ than to any other
root exactly when it lies in the corresponding open sector in
\eqref{eq:voronoi}. These sectors, their boundary rays, and the
origin exhaust $\C$, proving the claimed basin description.
\end{proof}

For $n>2$, the origin is a multiple zero of $p'(z)=nz^{n-1}$,
so the simple-critical-point hypothesis of the next section does
not apply. Nevertheless, Theorem~\ref{thm:unity} completely
describes the basins for this family. The escape map $E$ moves a
seed placed exactly at the origin. A nonzero seed on a boundary
ray never reaches the origin in finitely many steps, since
$Q(z)>0$ there. Its $E$-orbit therefore agrees with its $F$-orbit
and still converges to zero.

\section{Global basin geometry of RNM for general polynomials}\label{sec:geometry}

We now prove that, when all nonroot critical points are simple,
almost every seed converges to a root under $F$, and the
root-basin boundaries are contained in countably many Lipschitz
curves. The argument begins near each nonroot critical point,
where $F$ contracts in one direction and expands in another.
Every orbit converging to that point eventually enters a local
stable curve. We then show that all iterated preimages of this
curve remain contained in countably many Lipschitz curves.
Consequently, the nonroot critical basins have planar measure
zero. Finally, we identify their union with the union of the
root-basin boundaries.

Proposition~\ref{prop:criticalgraph} constructs the local
stable curve directly from the behavior of RNM near a
simple nonroot critical point.
Lemma~\ref{lem:pullback} shows that taking preimages
preserves containment in countably many Lipschitz curves.
Theorem~\ref{thm:geometry} combines these facts to establish
almost-everywhere root convergence and describe the
root-basin boundaries.
Corollary~\ref{cor:smooth} gives a covering by continuously
differentiable curves when the maximizing index at each
nonroot critical point is unique.
Finally, Proposition~\ref{prop:genericF} explains why
critical-point escapes are unnecessary for almost every
seed under the simplicity hypothesis, and how a small
linear perturbation ensures simple roots and critical points.

A \emph{Lipschitz curve} is a set parametrized by a map $\gamma:I\to\C\simeq\R^2$, where $I\subset\R$ is a bounded interval and, for some finite constant $L$, $|\gamma(s)-\gamma(t)|\leq L|s-t|$ for $s,t\in I$. In plain terms, changing the parameter by an amount $h$ moves the point by at most $L|h|$. Such a curve may have corners; differentiability is not required. Each curve has finite length, at most $L$ times the length of $I$, and has planar measure zero and Hausdorff dimension at most one. The same measure and dimension bounds hold for a countable union of these curves, although the union need not have finite total length or be locally smooth.

\subsection{Local stable curves}

We first show that every orbit converging to a simple nonroot
critical point eventually enters a local Lipschitz curve.
The argument uses the contracting and expanding directions
of $F$ at that point. Since the maximizing index in $A(z)$
may change nearby, we apply a stable-manifold theorem that
requires only a small Lipschitz remainder.

\begin{proposition}[Local structure at a simple nonroot critical point]
\label{prop:criticalgraph}
Suppose $p'(c)=0$, $p(c)\ne0$, and $p''(c)\ne0$.
There is a local Lipschitz curve $W_c$ through $c$, which
is a graph over an interval in the contracting direction,
such that $F(W_c)\subset W_c$ and every point of $W_c$
converges to $c$ under iteration of $F$.
Every orbit converging to $c$ eventually enters $W_c$.
Moreover, $W_c$ has Hausdorff dimension one.
These conclusions hold even if several normalized derivatives
attain the maximum $A(c)$.
\end{proposition}

\begin{proof}
Write $h(z)=1/(9A(z)^2)$ and $h_c=h(c)$.
Since $A$ is locally Lipschitz and bounded away from zero,
so is $h$. Taylor expansion gives
\begin{equation}\label{eq:critlinear}
 F(c+w)-c=Lw+R(w),\qquad
 Lw=w-\delta\overline w,\qquad
 \delta=h_c p(c)\overline{p''(c)},
\end{equation}
where $R(0)=0$.

We need the Lipschitz constant of $R$ to tend to zero
as its neighborhood shrinks. To verify this, write
\[
 p(c+w)\overline{p'(c+w)}
 =p(c)\overline{p''(c)}\,\overline w+H(w).
\]
On a ball of radius $r$, the polynomial remainder satisfies
$\sup|H|=O(r^2)$ and $\operatorname{Lip}(H)=O(r)$.
Therefore
\[
 R(w)=-(h(c+w)-h_c)p(c)\overline{p''(c)}\,\overline w
       -h(c+w)H(w)
\]
has $\operatorname{Lip}(R)=O(r)$: use
$\sup|h(c+w)-h_c|=O(r)$, the bounded local Lipschitz
constant of $h$, and the product estimate for Lipschitz
constants.

The real linear map $L$ is symmetric. Since
$|p(c)|\leq A(c)$ and $|p''(c)|\leq2A(c)$, we have
$0<|\delta|\leq2/9$. Thus its eigenvalues are
$\lambda_s=1-|\delta|\in(0,1)$ and
$\lambda_u=1+|\delta|>1$.
In orthogonal real coordinates centered at $c$, the map
therefore takes the form
$\Psi(x,y)=(\lambda_s x,\lambda_u y)+f(x,y)$,
where $f(0)=0$ and the Lipschitz constant of $f$ on a
sufficiently small square $Q=[-r,r]^2$, in the maximum
norm, is arbitrarily small.

To apply the invariant-manifold theorem, let $P_Q$ clip
each coordinate to $[-r,r]$. This map is $1$-Lipschitz
in the maximum norm, so $\widetilde f=f\circ P_Q$
extends $f$ to the plane with the same Lipschitz bound.
The global invariant-manifold theorem
\cite[Theorem~2.3(a),(e), arXiv version~1]{Bortolan2020},
applied with spectral separator $1$, gives a stable
Lipschitz graph through the origin for the extended map
$\widetilde\Psi(x,y)=(\lambda_s x,\lambda_u y)
+\widetilde f(x,y)$.
This graph consists exactly of the points whose extended
orbits converge to the origin.

For a sufficiently small Lipschitz bound $\kappa$ on $f$,
the graph has Lipschitz constant at most one and
$\lambda_s+\kappa<1$.
Its stable coordinate then contracts by a factor at most
$\lambda_s+\kappa$.
Restricting the graph to $|x|<\rho<r$ gives a
forward-invariant curve inside $Q$, where the original
and extended maps agree. Returning to the original
coordinates gives $W_c$.

Any $F$-orbit converging to $c$ eventually stays in this
square. Its tail consequently lies on the extended
stable graph and eventually enters $W_c$.
Finally, $W_c$ is a Lipschitz image of an interval,
and its orthogonal projection onto the contracting
direction contains that interval. These two facts give
Hausdorff dimension exactly one.
\end{proof}

\subsection{Pulling curves back through the global map}
The local stable curve contains only the late part of an orbit approaching a critical point. To describe the whole basin, we must also control all its preimages. For a set $S$, write $F^{-1}(S)=\{z:F(z)\in S\}$; this notation denotes a set preimage and does not assert that $F$ is invertible. The following argument isolates all changes of formula and all singular Jacobians in algebraic sets of dimension at most one.

\begin{lemma}[Countable-curve pullback]\label{lem:pullback}
If $S\subset\C$ is contained in countably many Lipschitz curves, then so is $F^{-1}(S)$.
\end{lemma}
\begin{proof}
For $j=0,\ldots,n$, define the real rational branch
\[
 \Phi_j(z)=z-\frac{p(z)\overline{p'(z)}}{9|a_j(z)|^2},\qquad a_j(z)\ne0.
\]
Let $X$ be the union of the following sets: all zeros of the $a_j$; all tie sets $|a_i|^2=|a_j|^2$, $i\ne j$; and the zero sets of the numerator polynomials of the real Jacobian determinants $\det D\Phi_j$. Each set is a proper real algebraic set. For the tie sets this follows because $a_i$ and $a_j$ have different degrees. For the determinant numerators, it suffices to check that none vanishes identically, as follows.

For $j=0$,
\[
 \Phi_0(z)=z-\frac19\overline{\frac{p'(z)}{p(z)}},
 \qquad \det D\Phi_0(z)\longrightarrow1\quad (|z|\to\infty).
\]
For $j=1$,
\[
 \Phi_1(z)=z-\frac19\frac{p(z)}{p'(z)},
 \qquad \det D\Phi_1(z)\longrightarrow\left(1-\frac1{9n}\right)^2>0.
\]
For $j\geq2$, the leading expression at infinity is
\[
 \Phi_j(z)=-d_j z|z|^{2j-2}+O(|z|^{2j-2}),
 \qquad d_j=\frac{n}{9\binom nj^2}>0.
\]
Differentiating these rational expressions, or directly taking their leading homogeneous terms, gives
\[
 \det D\Phi_j(z)=(2j-1)d_j^2|z|^{4j-4}+O(|z|^{4j-5}),
\]
which is not identically zero. Denominators are nonzero for all sufficiently large $|z|$, so these expansions are valid uniformly in direction.

A proper real algebraic subset of the plane admits a finite decomposition into smooth semialgebraic arcs and points; exhausting the arcs by compact parameter intervals shows that it is contained in countably many Lipschitz curves. This elementary consequence of semialgebraic stratification is discussed in \cite{BCR}. Hence $X$ is contained in countably many such curves.

At each point outside $X$, exactly one $|a_j|$ is maximal and nonzero, and $F=\Phi_j$ on a neighborhood with nonzero real Jacobian. The inverse function theorem provides a smooth local inverse. Choose countably many such charts covering $\C\setminus X$. Cover each chart image by countably many balls whose closures are compact and lie inside that image. On each ball the inverse has bounded derivative and is Lipschitz; their preimages still cover $\C\setminus X$.

Within one chart, write a covering curve of $S$ as $\gamma(I)$ for a Lipschitz map on a bounded interval. The open set of parameters mapping into the image ball is a countable union of intervals. On each such interval the composition of $\gamma$ with the local inverse is Lipschitz; compact exhaustion, and inclusion of any parameter endpoints, give countably many Lipschitz curves covering the relevant preimage. Countably many charts and covering curves therefore cover $F^{-1}(S)\setminus X$. Its part in $X$ is already covered. This proves the lemma.
\end{proof}

\begin{theorem}[Basin boundaries under simple nonroot critical points]\label{thm:geometry}
Suppose every nonroot critical point of $p$ is simple as a zero of $p'$. For $c\in\Crit(p)$, put $\mathcal B(c)=\{z:F^{\circ t}(z)\to c\}$. Then:
\begin{enumerate}[label=(\roman*)]
\item Each $\mathcal B(c)$ is contained in countably many Lipschitz curves, has planar Lebesgue measure zero, and has Hausdorff dimension exactly one.
\item The basins of the distinct roots are disjoint open sets with dense union and full planar measure. Moreover,
\begin{equation}\label{eq:boundaryidentity}
 \bigcup_{p(\theta)=0}\partial\mathcal A(\theta)
 =\bigcup_{c\in\Crit(p)}\mathcal B(c).
\end{equation}
\item If $p$ has at least two distinct roots, every root-basin boundary has Hausdorff dimension exactly one and is contained in countably many Lipschitz curves. If $p$ has only one distinct root, its basin is all of $\C$ and its boundary is empty.
\end{enumerate}
\end{theorem}
\begin{proof}
By Proposition~\ref{prop:criticalgraph}, there is a local Lipschitz graph $W_c\subset\mathcal B(c)$ that every orbit converging to $c$ eventually enters. Therefore
\begin{equation}\label{eq:Bpullback}
 \mathcal B(c)=\bigcup_{m\geq0}(F^{\circ m})^{-1}(W_c).
\end{equation}
Induction using Lemma~\ref{lem:pullback} proves countable-curve containment and dimension at most one. Inclusion of $W_c$ proves dimension at least one. Lipschitz curves have planar measure zero, so (i) follows.

Theorem~\ref{thm:global} partitions the plane into the root basins and the finitely many nonroot critical basins. By Theorem~\ref{thm:open}, the root basins are open; they are disjoint because each orbit has a unique limit. Their complement has measure zero by (i), so their union is dense and has full measure. No point of the boundary of a root basin lies in any root basin: an open neighborhood in that basin or in another root basin would contradict the boundary condition. This proves one inclusion in \eqref{eq:boundaryidentity}. Conversely, a point of a critical basin lies outside every root basin and, by density and finiteness of the collection, in the closure of at least one root basin. It therefore lies on its boundary, proving the reverse inclusion.

Every boundary in (iii) is contained in the right side of \eqref{eq:boundaryidentity}, hence in countably many Lipschitz curves and has dimension at most one. To prove the lower bound, choose a disk about a root $a$ contained in $\mathcal A(a)$ and a disk about a different root $b$ contained in $\mathcal A(b)$. The rays from $a$ meeting the second disk span a nonempty open interval of directions. Every such ray crosses $\partial\mathcal A(a)$ outside the first disk. Radial projection $z\mapsto(z-a)/|z-a|$ is Lipschitz outside the first disk and maps these boundary points onto that interval. Since Lipschitz maps cannot increase Hausdorff dimension, $\dimH\partial\mathcal A(a)\geq1$. Finally, if there is only one distinct root, $p$ is a constant times a power of $z-\theta$, and has no nonroot critical point. Theorem~\ref{thm:global} then gives $\mathcal A(\theta)=\C$.
\end{proof}

\begin{corollary}[A smoother local hypothesis]\label{cor:smooth}
Under the assumptions of Theorem~\ref{thm:geometry}, suppose also that at every nonroot critical point $c$ the maximum $A(c)$ is attained at a unique index. Then the critical basins and root-basin boundaries are contained in countably many $C^1$ curves.
\end{corollary}
\begin{proof}
The unique maximizer remains unique in a neighborhood of the critical point, where its value is nonzero. Thus $F$ is smooth there and has the hyperbolic derivative \eqref{eq:critlinear}. The classical local stable-manifold theorem \cite{HPS} supplies a smooth stable arc. In Lemma~\ref{lem:pullback}, the exceptional set has smooth arc strata, and the local inverses outside it are smooth. The same countable pullback argument with $C^1$ parametrizations proves the assertion.
\end{proof}

We use \emph{non-fractal} precisely for the proved countable Lipschitz-curve containment and Hausdorff dimension one of nonempty root-basin boundaries. These assertions allow corners and countably many accumulating arcs; they assert neither finite total length, a locally finite arc decomposition, nor a box-counting dimension bound. For $z^n-1$, Theorem~\ref{thm:unity} gives the stronger description by finitely many straight rays.

\begin{proposition}[Almost-everywhere sufficiency of $F$
and linear perturbations]\label{prop:genericF}
Let $p$ be a polynomial of degree $n\geq2$.
\begin{enumerate}[label=(\roman*)]
\item If every nonroot critical point of $p$ is simple,
then, for almost every initial seed, the $F$- and $E$-orbits
coincide and converge to a root of $p$.

\item There exists $\delta_0>0$ such that, whenever
$0<|\delta|<\delta_0$, the polynomial $p_\delta(z)=p(z)+\delta z$ has only simple roots and simple critical points.
Consequently, for almost every initial seed, its continuous
RNM iteration converges to a simple root and eventually
satisfies Smale's pointwise $\alpha$-condition, allowing
the permanent switch to quadratic Newton convergence
described in Theorem~\ref{thm:switch}.
\end{enumerate}
\end{proposition}

\begin{proof}
For~(i), Theorem~\ref{thm:geometry} shows that the union
of the root basins of $F$ has full planar measure.
An $F$-orbit in a root basin cannot encounter a nonroot
critical point, because $F$ would fix that point thereafter.
Since $E$ differs from $F$ only at nonroot critical points,
the two orbits coincide and have the same root limit.

For~(ii), $p_\delta'=p'+\delta$ and $p_\delta''=p''$. A nonsimple critical point $c$ would satisfy $\delta=-p'(c)$ and $p''(c)=0$. A multiple root would satisfy $p(c)+\delta c=0$ and $p'(c)+\delta=0$, hence $\delta=-p'(c)$ and $p(c)-cp'(c)=0$.
Thus all exceptional parameters belong to the set
\[
 D_p=
 \{-p'(c):p''(c)=0\}
 \,\cup\,
 \{-p'(c):p(c)-cp'(c)=0\}.
\]
This set is finite. Indeed, $p''$ is a nonzero polynomial,
and, writing $p(z)=a_nz^n+\cdots$ with $a_n\ne0$, the
polynomial $p-zp'$ has degree $n$ and leading coefficient
$(1-n)a_n\ne0$.

For example, take $\delta_0=\tfrac12\min\bigl(\{1\}\cup\{|d|:d\in D_p,\ d\ne0\}\bigr)>0$. This explicit choice excludes every exceptional parameter from $0<|\delta|<\delta_0$, so the roots and critical points of $p_\delta$ are simple there.

Applying part~(i) to $p_\delta$ now gives convergence
of its continuous RNM iteration to a simple root for
almost every seed. Along any such orbit,
$\beta_{p_\delta}(z_t)\to0$, while
$\gamma_{p_\delta}(z_t)$ is bounded for all sufficiently
large $t$. Consequently, $\alpha_{p_\delta}(z_t)\to0$.
The pointwise certification test therefore eventually
succeeds, and Theorem~\ref{thm:switch} justifies the
permanent Newton transition.
\end{proof}

Part~(ii) concerns roots of the perturbed polynomial $p_\delta$. It does not certify roots of $p$ itself. A subsequent Newton switch for $p$ requires its own pointwise test; such certification need not be available when the relevant root of $p$ is multiple.

\section{Implementation and polynomiography}\label{sec:figures}
\subsection{Arithmetic cost and stopping criteria}\label{sec:implementation}
Horner's rule evaluates $p$ and $p'$ in $O(n)$ arithmetic operations. Computing all normalized derivatives as the coefficients of the Taylor shift $p(z+w)$ costs $O(n^2)$ operations and $O(n)$ storage by a direct recurrence, or $O(n\ln n)$ arithmetic operations by a fast Taylor shift \cite[Section~2.2, algorithm~F]{FastTaylor1997}. Computing $A(z)$ from these coefficients takes $O(n)$ further operations. The fast bound concerns arithmetic operations and does not by itself guarantee floating-point stability.

For the periodic hybrid $H_T$, an accepted Newton trial between RNM updates uses $O(n)$ arithmetic operations, provided certification is not also evaluated on that iteration. A forced RNM update, or a rejected trial followed by the direct RNM fallback, costs $O(n^2)$. Pointwise evaluation of $\alpha_p$ from all Taylor coefficients has the same direct quadratic cost; checking it at RNM-update indices reuses those coefficients, as discussed after Theorem~\ref{thm:switch}. For a run stopped by the stationarity test, \eqref{eq:hybridcomplexity} gives the coarse direct arithmetic bound $O(n^2T A_0^2|p(z_0)|^2/\varepsilon^2)$ as $\varepsilon\downarrow0$. The fast Taylor shift replaces the factor $n^2$ in this coarse bound by $n\ln n$. These are arithmetic-operation bounds, not bit-operation bounds or degree-uniform bounds for certified root error.

Multiplying $p$ by a nonzero complex constant does not change $F$, $E$, $N$, the decisions of $H_T$ with the same initial counter, or $\alpha_p$. Indeed, $A(z)$ scales by the modulus of that constant, $p\overline{p'}$ scales by its squared modulus, and the factors in \eqref{eq:escape} cancel in the same way. In contrast, $g=|pp'|$ scales by the squared modulus. A raw tolerance on $g$ therefore has a scale-dependent interpretation and is a stationarity test, not a root-error certificate. A residual test $|p(z)|\leq\tau$ likewise requires conditioning information before it can be interpreted as an absolute root-error bound. Equation~\eqref{eq:certstop} is a separate certificate with a specified root-error meaning.

The exact escape rule distinguishes $p'(z)=0$ from $p'(z)\ne0$. A floating-point test such as $|p'(z)|<\tau$ defines a different algorithm unless a further error analysis justifies the replacement. Similarly, a computed value of $\alpha_p$ is a mathematical certificate only when its upper bound has been reliably established, for example by exact or validated arithmetic.

\subsection{Basin images and iteration coloring}
Figure~\ref{fig:basins} compares Newton and RNM basins.
The RNM polynomiographs iterate $F$ on a $601\times601$
uniform grid in $[-2,2]^2$, with at most $4000$ updates
and residual tolerance $10^{-8}$. Colors distinguish the
approximate roots reached; Figure~\ref{fig:art} adds
iteration-count shading. Light-gray pixels indicate that
no root was identified within these stopping conditions;
they do not by themselves indicate failure of root convergence.
For generation and coloring methods, see \cite{Invitation2025}.

\begin{figure}[p]
 \centering
 \includegraphics[width=0.70\textwidth]{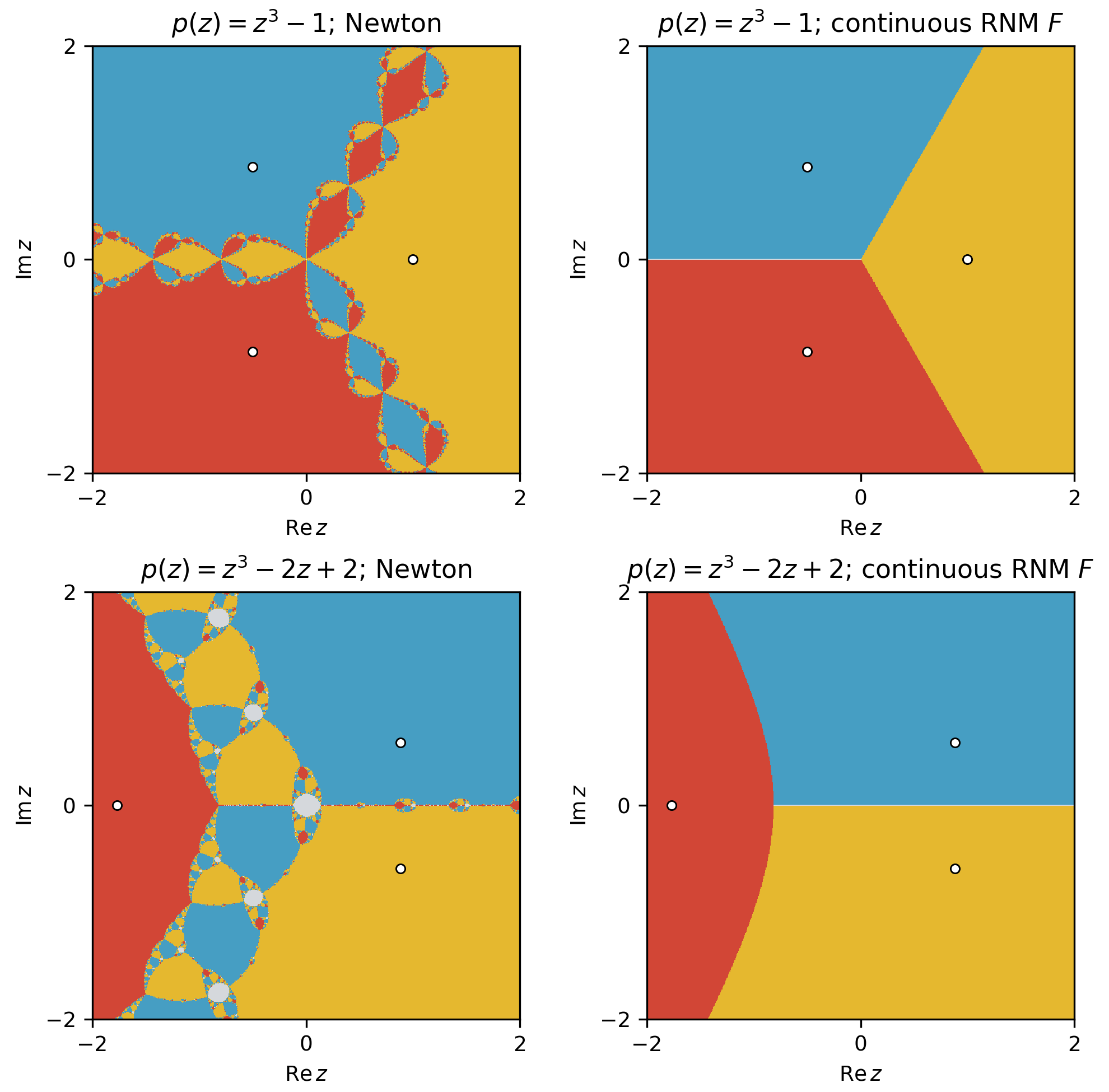}
 \caption{Newton basins (left) and RNM basins (right) for $z^3-1$ (top) and $z^3-2z+2$ (bottom). White circles mark roots; light-gray pixels indicate that no root was identified within the stopping conditions. The RNM geometry follows from Theorems~\ref{thm:unity} and~\ref{thm:geometry}, respectively. The second polynomial has the attracting Newton cycle $0\leftrightarrow1$.}\label{fig:basins}
 \medskip
 \includegraphics[width=0.85\textwidth]{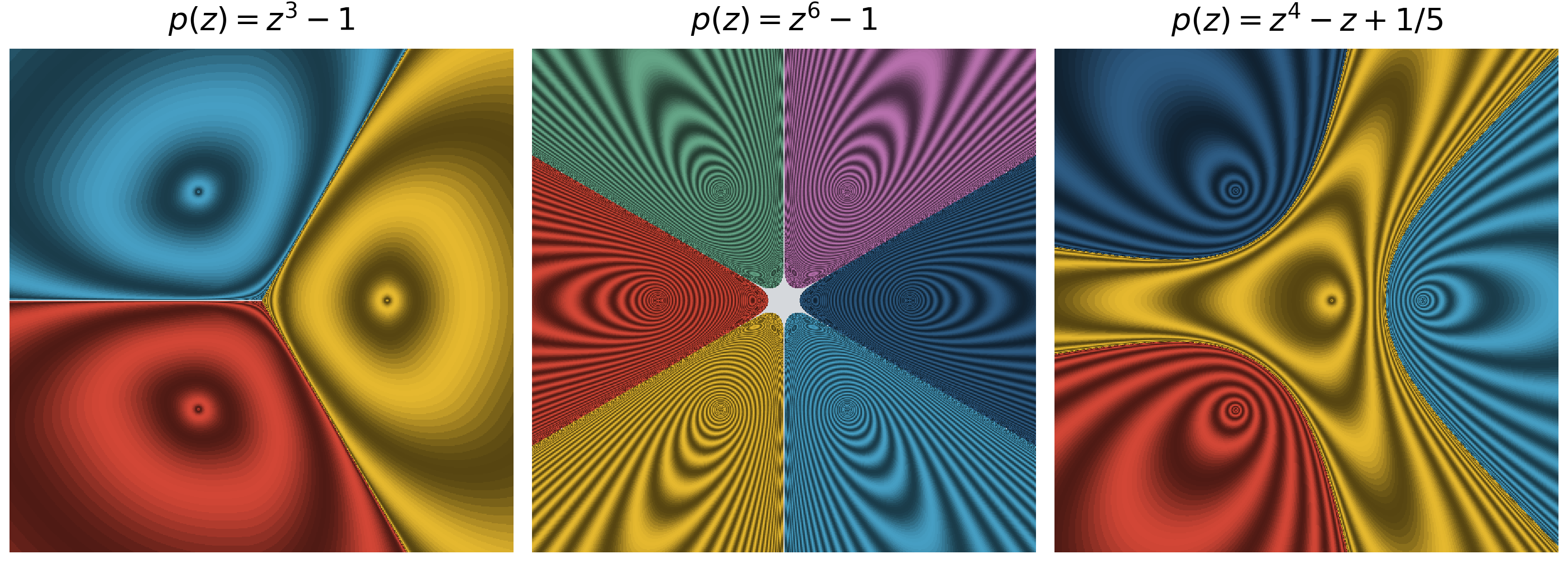}
 \caption{RNM polynomiographs colored by root and iteration count: $z^3-1$, $z^6-1$, and $z^4-z+1/5$. The first two have the exact Voronoi basins of Theorem~\ref{thm:unity}. The quartic has simple nonroot critical points, so Theorem~\ref{thm:geometry} proves its one-dimensional boundary structure. Internal bands reflect stopping time, not additional basin boundaries.}\label{fig:art}
\end{figure}

\section{Concluding remarks}\label{sec:conclusion}
RNM combines an explicit polynomial-modulus descent step with global convergence and precise basin geometry. Descent keeps every orbit in a compact sublevel set and forces $|pp'|\to0$; vanishing steps and finitely many possible limits then give convergence to a single root or critical point. The escape map leaves exact nonroot critical points, while the continuous map supports the open-basin and boundary theorems. Under simple nonroot critical points, almost every seed converges to a root, even when the maximum defining $A(z)$ has ties.

Periodic Newton acceleration retains convergence and the stationarity bound. For orbits approaching a simple root, Smale certification gives a finite, accuracy-independent preparation followed by quadratic Newton convergence. The cubic initialization makes this preparation quantitatively explicit, with a conservative bound retained by RNM-guided Newton. The brief computations illustrate possible acceleration and its dependence on the seed; they do not supply a uniform runtime comparison.

For practical experimentation and polynomiography, one may also consider \emph{truncated RNM}: for $0\leq m\leq n$, put $A_m(z)=\max_{0\leq j\leq m}|p^{(j)}(z)|/j!$ and $F_m(z)=z-p(z)\overline{p'(z)}/(9A_m(z)^2)$ at nonroots, and fix roots. Here $A_0(z)=|p(z)|$ is distinct from the sublevel bound $A_0$ in \eqref{eq:sublevel}. The case $m=n$ gives $F$; smaller $m$ uses fewer derivatives and may take larger steps, producing further polynomiographs \cite{Invitation2025}. The full-map theorems do not automatically extend to these variants. Sharper preparatory bounds, efficient derivative estimates, and basin geometry at general degenerate nonroot critical points remain natural questions.

\section*{Acknowledgments}
The author acknowledges the assistance of ChatGPT (OpenAI)
with mathematical review, proof development, exposition,
figure-generation code, and numerical implementation and
analysis. The author takes full responsibility for the
content of the article.

\FloatBarrier

\end{document}